\documentclass[12pt]{amsart}

\usepackage{lmodern,amsmath,amssymb,mathtools,microtype}
\usepackage[a4paper,margin=23mm]{geometry}
\usepackage{xurl}
\usepackage[hidelinks]{hyperref}
\newtheorem{theorem}{Theorem}
\newtheorem{lemma}[theorem]{Lemma}
\newtheorem{proposition}[theorem]{Proposition}
\newtheorem{corollary}[theorem]{Corollary}
\newtheorem{hypothesis}[theorem]{External input}
\theoremstyle{remark}\newtheorem{remark}[theorem]{Remark}
\numberwithin{equation}{section}
\newcommand{\R}{\mathbb R}
\newcommand{\Z}{\mathbb Z}
\newcommand{\supp}{\operatorname{supp}}
\newcommand{\1}{\mathbf 1}
\newcommand{\ip}[1]{\langle #1\rangle}

\title[The square function estimate for the cone]{The square function estimate for the cone in four dimensions}
\author{Faruk Temur}
\date{October 8, 2026}
\hypersetup{pdftitle={A square function estimate for the cone in four dimensions},
 pdfauthor={Faruk Temur},pdfsubject={A conditional cone estimate and its elliptic consequences}}
\begin{document}
\begin{abstract}
We obtain the $L^3$ reverse square function estimate for the light cone in four dimensions from the critical local-smoothing and round-cone propagation statements in the September 2026 OpenAI manuscript. We state these inputs, together with the plate maximal theorem of Di Plinio and Parissis, in a preliminaries section. The proof passes from geometric concentration of packet coefficients to regularity of the analyzed wave, and then from free waves to arbitrary functions supported near the cone. As a consequence, we obtain the corresponding $L^3$ estimate for the sphere and elliptic paraboloids in three dimensions. 
\end{abstract}
\maketitle

\section{Introduction}
Let
\[
 \Gamma=\{(\xi,\tau)\in\R^3\times\R:\tau=|\xi|,\quad1\leq|\xi|\leq2\}.
\]
For $R\geq2$, put $h=R^{-1}$. Choose a maximal $\sqrt h$-separated set
$\{\omega_\theta\}\subset S^2$. For a fixed constant $C_0\geq2$, define the sectors
\begin{equation}\label{sectorintro}
 \theta=\left\{(\xi,\tau):1\leq|\xi|\leq2,\quad
 \left|\frac\xi{|\xi|}-\omega_\theta\right|\leq C_0\sqrt h,
 \quad|\tau-|\xi||\leq C_0h\right\}.
\end{equation}
Each sector is contained in a rectangular frequency plank with side lengths
$1,\sqrt h,\sqrt h,h$. The family has bounded overlap. We give the construction, including the corresponding smooth partition, in Section~\ref{geometry}.

\begin{theorem}[Cone square function estimate]\label{main}
Assume the external local-smoothing and round-cone inputs stated in
Section~\ref{preliminaries}. For every $\varepsilon>0$ there is a constant
$C_{\varepsilon,C_0}$ such that, for every $R\geq2$ and every family of
Schwartz functions $f_\theta:\R^4\to\mathbb C$ satisfying
$\supp\widehat f_\theta\subset\theta$, with $\theta$ as in
\eqref{sectorintro},
\begin{equation}\label{target}
 \left\|\sum_\theta f_\theta\right\|_{L^3(\R^4)}
 \leq C_{\varepsilon,C_0}R^\varepsilon
 \left\|\left(\sum_\theta|f_\theta|^2\right)^{1/2}\right\|_{L^3(\R^4)}.
\end{equation}
The constant is independent of $R$, the choice of the separated angular net,
and the functions. Fixed changes of the radial interval and fixed enlargements
of the canonical planks are allowed, with a corresponding change of constant.
\end{theorem}

The functions in this statement are arbitrary within their sectors. In particular,
they need not be solutions of a wave equation, and need not arise from one
prescribed partition of a single function. This distinction matters when we
pass from a free-wave estimate to \eqref{target}.

The right side is the norm of the \emph{pointwise} square function. For $p\geq2$,
Minkowski's inequality gives
\[
 \left\|\left(\sum_\theta|f_\theta|^2\right)^{1/2}\right\|_p
 \leq \left(\sum_\theta\|f_\theta\|_p^2\right)^{1/2}.
\]
Consequently an estimate with the latter expression on the right does not,
by itself, give the theorem. At $p=2$, bounded overlap of the frequency
supports gives the reverse square function estimate by Plancherel. The
content here is its extension to $p=3$ with an arbitrarily small power loss.

\subsection{History}
The connection between square functions, local smoothing, and maximal
operators was developed by Mockenhaupt, Seeger, and Sogge
\cite{MSS92,MSS93}, following Sogge's formulation of the local-smoothing
problem \cite{Sogge}. Wolff \cite{Wolff} established local-smoothing type
estimates for large $p$ using wave packets and geometric incidence estimates.
Bourgain and Demeter \cite{BD} subsequently proved sharp $\ell^2$ decoupling
for the cone. Decoupling is a major part of this history, but the distinction
between the two right sides above remains relevant to the present question.

Guth, Wang, and Zhang \cite{GWZ} proved the sharp $L^4$ reverse square function
estimate for the cone in $\R^3$, and deduced local smoothing in $2+1$
dimensions. Their frequency sectors have dimensions
$1,R^{-1/2},R^{-1}$. The cone considered in this article lies in $\R^4$;
it has two curved directions, and the estimate studied here is the $L^3$
analogue. Thus the ambient dimension and the exponent both change.

Di Plinio and Parissis \cite{DPP} proved sharp estimates for maximal subspace
averages, including the codimension-one plate estimate used below. The
September 2026 manuscript \cite{OAI} states critical local smoothing for the
three-dimensional wave equation and a propagation estimate controlled by
round and hyperquadric regularity tests. We use these statements as specified
external inputs. The present article proves the transfer to arbitrary
thickened-cone functions; it does not supply an independent proof of those
statements.

\subsection{The tools used in the proof}
There are two uses of the local-smoothing article. Its scalar local-smoothing
estimate, Theorem 1.1, is combined with randomization and the Di Plinio--Parissis
plate maximal theorem to obtain a positive averaging estimate for physical
cone planks. Its round-cone theorem, Theorem 9.1, controls the growth of
analyzed wave energy in terms of its initial concentration. For this second
use we also need the tests in Definitions 8.1--8.2 and their enlargement
property, Lemma 8.4. We state all of these below in the forms used here.

The Di Plinio--Parissis input is specifically the codimension-one case of
Theorem C in \cite{DPP}. It concerns averages over spatial plates in $\R^3$,
not spacetime cone planks. Interpolation gives an $L^3$ bound with logarithmic
loss. Local smoothing then converts this spatial estimate into the spacetime
averaging estimate needed to recover the original square function.

The remainder of the argument has three steps. We split an arbitrary packet
array according to its concentration in the regularity tests. We compare that
concentration with the energy of the analyzed wave and apply the propagation
input. Finally, a Fourier-series expansion and the positive averaging estimate
return us to the original $f_\theta$. This last step retains cancellation
inside each sector.

\section{The cone and its plank decompositions}\label{geometry}
We use the Fourier transform
$\widehat f(\zeta)=\int e^{-iX\cdot\zeta}f(X)\,dX$.
Fixed powers of $2\pi$ are absorbed into normalizations. The notation
$A\lesssim B$ means $A\leq CB$ with a constant independent of $R$ or $h$.
A subscript records additional dependence. We always keep the angular chart,
cutoffs, and enlargement constants fixed.

\subsection{The angular partition}
A maximal $\sqrt h$-separated set on $S^2$ has $O(h^{-1})$ points, and its
$\sqrt h$-balls cover the sphere. Separation implies bounded overlap of the
$C\sqrt h$-balls for each fixed $C$. Choose smooth nonnegative bumps on these
balls and divide them by their sum. This gives a partition
\[
 \sum_\theta\psi_\theta(\omega)=1,\qquad
 \supp\psi_\theta\subset\{\omega:|\omega-\omega_\theta|\leq C\sqrt h\},
 \qquad |\nabla^k\psi_\theta|\lesssim_k h^{-k/2}.
\]
Thus a function supported in the thickened cone can be decomposed by
\[
 \widehat f_\theta(\xi,\tau)
 =\psi_\theta(\xi/|\xi|)\widehat f(\xi,\tau),
\]
with harmless radial and normal cutoffs if needed. The word ``partition''
refers to this smooth partition of unity. The supporting planks form a
finite-overlap cover, rather than a disjoint rectangular partition.

\subsection{Why the frequency planks have these dimensions}
Fix a sector centre $\omega$, and choose orthonormal vectors $e_1,e_2$ in
$\omega^\perp\subset\R^3$. An orthonormal frame in frequency space is
\[
 g_\omega=\frac{(\omega,1)}{\sqrt2},\qquad
 E_j=(e_j,0)\ (j=1,2),\qquad
 n_\omega=\frac{(\omega,-1)}{\sqrt2}.
\]
The vector $g_\omega$ points along a generator of the cone; $E_1,E_2$ are
its two angular tangent directions; and $n_\omega$ is normal to the cone.
Put $q_\omega=(\tfrac32\omega,\tfrac32)$. If
$(\xi,\tau)=(r\omega',r+s)$ belongs to \eqref{sectorintro}, then
\begin{align*}
 |((\xi,\tau)-q_\omega)\cdot g_\omega|&\lesssim1,\\
 |((\xi,\tau)-q_\omega)\cdot E_j|&\lesssim\sqrt h,\\
 |((\xi,\tau)-q_\omega)\cdot n_\omega|
 &=2^{-1/2}|r(\omega'\cdot\omega-1)-s|\lesssim h.
\end{align*}
For the last inequality use
$1-\omega'\cdot\omega=\tfrac12|\omega'-\omega|^2$.
The sector therefore lies in an affine rectangular plank $\Theta_\omega$
with the asserted dimensions. Conversely, a sufficiently small fixed
rectangular plank of this shape, on a fixed positive radial band, lies in a
fixed enlargement of the sector. The two descriptions are interchangeable
up to fixed subdivisions and enlargements. Planks extending to the vertex
are not used.

This calculation explains the square root scale: an angular displacement
of size $\sqrt h$ produces a normal error of size $h$. There is no
corresponding quadratic error along the radial generator, so its frequency
length remains comparable to one.

\subsection{Physical dual planks}
The physical dual of $\Theta_\omega$ is a translate of a box satisfying
\[
 |X\cdot g_\omega|\lesssim1,\qquad
 |X\cdot E_j|\lesssim h^{-1/2},\qquad
 |X\cdot n_\omega|\lesssim h^{-1}.
\]
These reciprocal side lengths are the scales supplied by the uncertainty
principle. For reference,
\[
\begin{array}{c|ccc}
 &\text{generator}&\text{two angular directions}&\text{normal}\\ \hline
 \text{frequency plank}&1&\sqrt h,\sqrt h&h\\
 \text{physical dual plank}&1&h^{-1/2},h^{-1/2}&h^{-1}
\end{array}
\]
The physical volume is comparable to $h^{-2}=R^2$. Its long direction is
$n_\omega$, so in spacetime it follows a ray with $x+t\omega$ constant.
This sign agrees with the propagator $e^{it|D|}$ in our Fourier convention.

For each fixed $\omega$, translates of one half-open dual box by the
corresponding rectangular lattice tile physical space. Different directions
have different tilings. Wave packets are smooth functions localized to
these boxes, with rapidly decreasing tails. They are not obtained by
multiplying by characteristic functions of the boxes: such multiplication
would destroy the desired frequency support. The precise smooth expansion
is given in Section~\ref{thickened}.

\subsection{Rescaling and null coordinates}\label{nullcoordinates}
Set $F(X)=f(RX)$. Then
$\widehat F(\zeta)=R^{-4}\widehat f(\zeta/R)$ and
$\|F\|_3=R^{-4/3}\|f\|_3$. Both sides of \eqref{target} have this same
factor. We may therefore work at frequency $h^{-1}$, with cone thickness
$O(1)$. The physical plank dimensions become
$h,\sqrt h,\sqrt h,1$, and their volume becomes $h^2$.
For the proof we may take $h$ dyadic: replacing $R$ by a neighbouring dyadic
number changes the cap sizes and overlaps by fixed factors only.

On a chart where $\tau+\xi_3\geq c\tau$, put
\[
 \eta=(\xi_1,\xi_2),\qquad \rho=\frac{\tau+\xi_3}{2},\qquad
 \sigma=\frac{\xi_3-\tau}{2}.
\]
The cone equation becomes
\begin{equation}\label{nullcone}
 \sigma=-\frac{|\eta|^2}{4\rho},\qquad \rho>0.
\end{equation}
If the original physical coordinates are $(x',x_3,t_{\rm old})$, the dual
coordinates are
\[
 b=x',\qquad y=-(x_3+t_{\rm old}),\qquad t=x_3-t_{\rm old}.
\]
Indeed the Fourier pairing is $b\cdot\eta-y\rho+t\sigma$.
A finite number of rotated charts covers the cone. On each one,
$h\rho$ lies in a fixed compact subinterval of $(0,\infty)$ and
$A=\eta/(2\rho)$ lies in a fixed bounded set.

The associated free evolution is
\begin{equation}\label{nullprop}
 \widehat{U_0(t)v}(\eta,-\rho)
 =e^{-it|\eta|^2/(4\rho)}\widehat v(\eta,-\rho).
\end{equation}
Its rays are $(b,y)=(B+tA,Y+t|A|^2)$. All changes just made are linear
changes of frequency variables or ordinary dilations. Their dual changes
multiply both sides of a square function estimate by the same determinant
factor.

\section{Preliminaries and external results}\label{preliminaries}
We state the external results in full in the specialized forms needed here.
The basic tools used in addition are Plancherel, H\"older, Minkowski,
Fourier series, and Khintchine's inequality. In the latter, independent
signs $\varepsilon_j=\pm1$ satisfy
\[
 \mathbb E\left|\sum_j\varepsilon_jz_j\right|^3
 \asymp\left(\sum_j|z_j|^2\right)^{3/2}.
\]
All norms below are with respect to Lebesgue measure unless a row measure
is displayed.

\subsection{The Di Plinio--Parissis plate maximal theorem}
For $0<a\leq\tfrac12$ and $\omega\in S^2$, let
\[
 \begin{split}
 T_a(\omega)&=\{z\in\R^3:|z-(z\cdot\omega)\omega|\leq1,
                  \ |z\cdot\omega|\leq a\},\\
 N_a u(x)&=\sup_{\omega\in S^2}\frac1{|T_a(\omega)|}
                    \int_{T_a(\omega)}|u(x-z)|\,dz.
 \end{split}
\]
These are two-dimensional plates thickened in their normal direction.
The weak $L^2$ norm is
$\|u\|_{2,\infty}=\sup_{\lambda>0}\lambda
 |\{x:|u(x)|>\lambda\}|^{1/2}$.

\begin{hypothesis}[Di Plinio--Parissis, Theorem C]\label{dppinput}
The codimension-one plate maximal operator in $\R^3$ satisfies
\begin{equation}\label{dppweak}
 \|N_a u\|_{2,\infty}\leq C\bigl(\log(2/a)\bigr)^{1/2}\|u\|_2.
\end{equation}
This is the case $d=2$, $n=3$ of \cite[Theorem C]{DPP}.
\end{hypothesis}

We call quoted results ``external inputs'' to distinguish them from lemmas
proved here; \eqref{dppweak} is the published Di Plinio--Parissis theorem.
It implies
\begin{equation}\label{dppstrong}
 \|N_a u\|_3\lesssim\bigl(\log(2/a)\bigr)^{1/3}\|u\|_3.
\end{equation}
Here is the interpolation argument. The part of $u$ bounded by $\lambda/2$
has maximal average at most $\lambda/2$. Apply \eqref{dppweak} to the
remaining part and integrate its distribution function:
\[
\begin{split}
 \|N_a u\|_3^3
 &=3\int_0^\infty\lambda^2|\{N_a u>\lambda\}|\,d\lambda\\
 &\lesssim\log(2/a)\int_0^\infty
           \int_{|u|>\lambda/2}|u(x)|^2\,dx\,d\lambda
 \lesssim\log(2/a)\|u\|_3^3.
\end{split}
\]
After an isotropic dilation by $\sqrt h$, the choice $a=\sqrt h$
gives the spatial plates of dimensions $\sqrt h,\sqrt h,h$.
Fixed enlargements and positive Schwartz majorants satisfy the same bound:
decompose a majorant into isotropically dilated plates with summable
coefficients. The finitely many large values of $h$ are harmless.

\subsection{Scalar local smoothing}
Let $J^s=(1-\Delta)^{s/2}$.
\begin{hypothesis}[Critical scalar local smoothing]\label{lsinput}
Theorem 1.1 of \cite{OAI} states that, for every $s>0$,
\begin{equation}\label{lsquoted}
 \|e^{it|D|}g\|_{L^3(\R^3\times[1,2])}
 \leq C_s\|J^s g\|_{L^3(\R^3)}.
\end{equation}
\end{hypothesis}

On a fixed annular band $c h^{-1}\leq|\xi|\leq C h^{-1}$ this gives
\begin{equation}\label{LS}
 \|e^{it|D|}g\|_{L^3(\R^3\times[-2,2])}
 \leq C_s h^{-s}\|g\|_3.
\end{equation}
To include times near zero, split positive times into dyadic intervals
$[a,2a]$. Simultaneous dilation of space and time gives a factor
$a^{1/3}(a/h)^s$ for $a\geq h$; these factors have a summable geometric
series. For $|t|\leq h$, the frequency-localized multiplier has a uniformly
integrable kernel after scaling by $h$. Negative times follow by conjugation.
This proves \eqref{LS} from the quoted formulation.

\subsection{Analysis of a free wave}
The second input from \cite{OAI} requires some definitions. They are included
because a bound for arbitrary packet coefficients cannot be substituted
for the source theorem without checking its regularity hypothesis.

At interval length $\ell\in[h,1]$, write
\[
 d_\ell=(h/\ell)^{1/2},\qquad q_\ell=(h\ell)^{1/2}.
\]
Choose $\psi\in C_c^\infty(\R^2)$ with $\int|\psi|^2=1$, and choose a
smooth radial cutoff $\chi$, equal to one on the fixed incoming band in
$h\rho$. For each angular parameter $A$, define
\begin{equation}\label{analyzer}
 (V_\ell v)(A,b,y)=\mathcal F^{-1}_{\eta,-\rho}
 \left[\chi(h\rho)\psi\left(\frac{\eta/(2\rho)-A}{d_\ell}\right)
                 \widehat v(\eta,-\rho)\right](b,y),
 \qquad d\nu_\ell=d_\ell^{-2}\,dA\,db\,dy.
\end{equation}
The triple $(A,b,y)$ is called a \emph{row}: it records the direction and
position of an analyzed piece of the wave. Plancherel and
$d_\ell^{-2}\int|\psi((a-A)/d_\ell)|^2\,dA=1$ give
\begin{equation}\label{isometry}
 \int|V_\ell v|^2\,d\nu_\ell=\|v\|_2^2
\end{equation}
on this band. This is the ``plateau analysis'' of \cite[Definition 8.1]{OAI}.
For an interval $I=[T,T+\ell]$, its row energy measure is
\[
 d\mu_I=|V_\ell U_0(T)v_0|^2\,d\nu_\ell.
\]

\subsection{The regularity tests}
At time $T$, length $\ell$, centre $(b_*,y_*,A_*)$, and width $r>0$, set
\begin{equation}\label{testcoordinates}
\begin{aligned}
 u&=(t-T)/\ell,&
 x&=\frac{b-b_*-\ell uA_*}{\ell r},\\
 z&=\frac{y-y_*-2A_*\cdot(b-b_*)+\ell u|A_*|^2}{\ell r^2},&
 V&=(A-A_*)/r.
\end{aligned}
\end{equation}
The normalized rays have velocities $(V,|V|^2)$. A \emph{round test}
imposes $|x|,|z|,|V|\leq1$ at $u=0$ and has thickness parameter $f=1$.
A \emph{hyperquadric test} imposes these same bounds and also
\begin{equation}\label{testFG}
 |F|\leq f,\quad |G|\leq f,\quad 0<f\leq1,
 \qquad
 \begin{cases}
 F=c+du+\alpha\cdot x+\beta z+k(uz-|x|^2),\\
 G=(\partial_u+V\cdot\nabla_x+|V|^2\partial_z)F,
 \end{cases}
\end{equation}
where
\begin{equation}\label{discriminant}
 |\alpha|^2-4\beta d+4kc=1,\qquad
 \max(|c|,|d|,|\alpha|,|\beta|,|k|)\leq100.
\end{equation}
Every test satisfies the uncertainty condition $rf\geq d_\ell$.
For fixed $0<\gamma<1$ its weight is
\begin{equation}\label{testweight}
 w(E)=r^4f^{1-\gamma}.
\end{equation}
A positive row measure $\mu$ has \emph{regularity constant} $\kappa$ if
$\mu(E)\leq\kappa w(E)$ for every test at the given length and time.
This is \cite[Definition 8.2]{OAI}.

The interpretation is useful. A round test collects nearby positions and
directions. A thin test collects rays close to a specified quadratic
hypersurface and nearly tangent to it: $F$ measures displacement and $G$
measures change along the ray. The discriminant normalization prevents the
polynomial from being arbitrarily small merely because all its coefficients
are small. The uncertainty condition excludes tests finer than the angular
resolution of the analysis.

\begin{hypothesis}[Enlargement of tests]\label{enlargedinput}
In the tests above, enlarge the normalized round bounds by fixed powers of
$P\geq2$, allow coefficients bounded by fixed powers of $P$ while keeping
discriminant one, and replace the two sublevels by $P^Cf$.
The resulting set can be covered by admissible tests whose total weight is
at most $C'P^{C''}r^4f^{1-\gamma}$. The constants depend on the specified
enlargement powers, not on the scale or the test. Consequently any positive
measure regular with constant $\kappa$ assigns it mass at most
$C'P^{C''}\kappa r^4f^{1-\gamma}$.
\end{hypothesis}

This is the geometric covering assertion in the proof of
\cite[Lemma 8.4]{OAI}; the lemma itself states the resulting mass bound.
The covering assertion applies to atomic measures as well as row energy
measures. For clarity, the covering uses round tests when $f$ is larger than
a fixed inverse power of $P$. Otherwise the discriminant identity gives a
nonzero transverse gradient. Subdivide into parabolic boxes of width
$P^{-D}$, centre at a witness, and normalize the polynomial there.
Its coefficients are then bounded and its relative thickness is at most a
fixed power of $P$ times $f/P^{-D}$. There are polynomially many boxes and
the absolute uncertainty width does not decrease. This gives the asserted
total cost. We will use the stated covering property as part of the external
input.

\subsection{The round-cone propagation input}
We now give the special case of the propagation theorem that is used in the
proof. Here and below a \emph{fixed polynomial range} means that some fixed
$B$ bounds the number of labels, the size of the position support, and the
relevant parameter ratios by $C h^{-B}$. The exponent $B$ is fixed before
$h$ tends to zero. Rapid tails mean that the discarded energy is at most
$C_Nh^N$ times the total energy for every prescribed $N$.

\begin{hypothesis}[Round-cone propagation]\label{propinput}
Let $v_0\in L^2(\R^3)$ have energy $e=\|v_0\|_2^2>0$, with $h\rho$ in
a fixed compact positive band and $\eta/(2\rho)$ in a fixed bounded set.
Suppose its initial row measure $|V_1v_0|^2d\nu_1$ is regular with constant
$\kappa$, and that the data form a uniform fixed polynomial family: the
ratios $e/\kappa$ and $\kappa/e$ have fixed polynomial bounds, the initial
row density relative to $e$ has a fixed polynomial bound, and the initial
row energy has rapid tails outside a fixed polynomial phase box.
For $j=0,\ldots,h^{-1}-1$ and spatial cubes $Q$ of side $h$, put
\[
 m_{j,Q}=\int_{\{(A,b,y):(b,y)\in Q\}}
          |V_hU_0(jh)v_0(A,b,y)|^2\,d\nu_h.
\]
For every $0<\gamma<1$ and every $\zeta>0$,
\begin{equation}\label{propquoted}
 h^{-1/2}\sum_{j,Q}m_{j,Q}^{3/2}
 \leq C_{\gamma,\zeta,B}h^{-100\gamma-\zeta}e\sqrt\kappa.
\end{equation}
The constant is uniform when all fixed setup constants and tail bounds are
held fixed. Polynomial truncation followed by rapid-tail summation is
understood if the spatial support is unbounded.
\end{hypothesis}

This is the free scalar specialization of \cite[Theorem 9.1]{OAI}.
At terminal length $h$ the angular aperture is $d_h=1$. A round test with
fixed width contains one spatial $h$-cube and the bounded set of all row
angles, so its weight is comparable to one. The assignments defining
$m_{j,Q}$ are disjoint at each time. These facts identify the displayed
functional with the bounded-weight case of the source theorem.

The source states an upper growth exponent $\Gamma_{\rm wav,round}\leq100\gamma$;
\cite[(10.22)]{OAI} explains its finite-scale use. If
\eqref{propquoted} failed on a family with fixed setup constants, a sequence
$h\downarrow0$ would have normalized growth exceeding $h^{-100\gamma-\zeta}$,
contradicting that exponent bound. This explains both the extra $\zeta$ and
the uniformity requirement. No contractions or Hilbert-valued extension of
the source theorem are needed here.

For the remainder of the article the local-smoothing and propagation
statements, including the enlargement property, are assumptions. All
passages from them to \eqref{target} are proved below.

\section{A conic averaging consequence of scalar local smoothing}\label{averaging}
Write $U(t)=e^{it|D|}$. We use the frequency-localized form
\eqref{LS} of External input~\ref{lsinput}, and the spatial maximal estimate
\eqref{dppstrong}. This section is the first of the two uses of the
local-smoothing article.

For $\omega\in S^2$, write $x=x_\perp+x_\parallel\omega$ and define the spatial plate
\[
 B_{\omega,h}=\{x:|x_\perp|\leq\sqrt h,\ |x_\parallel|\leq h\}.
\]
Its volume is comparable to $h^2$. The associated spacetime plate is
\begin{equation}\label{plate}
 P_{\omega,h}=\{(x,t):|t|\leq1,\ x+t\omega\in B_{\omega,h}\}.
\end{equation}
Fixed enlargements of these sets are allowed.

We use the established codimension-one plate maximal theorem of Di Plinio and Parissis \cite[Theorem C and Section 7]{DPP}. At $p=3$, it implies
\begin{equation}\label{spatialmax}
 \left\|\sup_\omega\frac{\1_{B_{\omega,h}}}{|B_{\omega,h}|}*|a|\right\|_3
 \lesssim (1+\log h^{-1})^{1/3}\|a\|_3.
\end{equation}
This is \eqref{dppstrong} after dilation; it also applies to the positive Schwartz majorants described in the preliminaries.

\begin{proposition}[Positive conic averaging]\label{avg}
Assume \eqref{LS}. For every finite family of nonnegative functions $G_\theta$ on $\R^4$ and directions $\omega_\theta$,
\begin{equation}\label{avg4}
 \left\|\sum_\theta K_{\theta,h}*G_\theta\right\|_{3/2}
 \leq C_\epsilon h^{-\epsilon}
 \left\|\sum_\theta G_\theta\right\|_{3/2},
\end{equation}
where $K_{\theta,h}$ is a normalized characteristic kernel of \eqref{plate}, or a fixed rapidly decreasing majorant adapted to it. Equivalently, the associated conic plate maximal operator has $L^3$ norm $O_\epsilon(h^{-\epsilon})$.
\end{proposition}

\begin{proof}
Choose $L^2$-normalized initial packets $p_{\theta,h}$ with Fourier support in a small box around $h^{-1}\omega_\theta$, of radial width $c h^{-1}$ and transverse widths $c h^{-1/2}$. The fixed constant $c>0$ is chosen sufficiently small. Their Fourier profiles can be chosen nonnegative. Taylor expansion of $|\xi|$ around the central ray gives
\begin{align}
 |U(t)p_{\theta,h}(x)|^2&\geq c_0h^{-2}\1_{B_{\omega_\theta,h}}(x+t\omega_\theta),
       &&|t|\leq1,\label{packetlower}\\
 |p_{\theta,h}(x)|^2&\leq C_Mh^{-2}
 (1+|x_\perp|/\sqrt h+|x_\parallel|/h)^{-M}.
       \label{packetupper}
\end{align}
Indeed the nonlinear phase over this box is $O(c^2)$ for $|t|\leq1$, and on the displayed spatial core the remaining normalized linear phase is $O(c)$. The integral therefore stays in a fixed open half-plane and has modulus bounded below by a constant times $h^{-1}$. Integration by parts at time zero gives \eqref{packetupper}.

Apply \eqref{LS} to finite randomized sums
\[
 g=\sum_{\theta,j}\varepsilon_{\theta,j}a_{\theta,j}
           p_{\theta,h}(\,\cdot-z_{\theta,j}).
\]
Khintchine's inequality and integration give
\begin{equation}\label{randomLS}
 \left\|\sum_{\theta,j}|a_{\theta,j}|^2
       |U(t)p_{\theta,h}(x-z_{\theta,j})|^2\right\|_{L^{3/2}_{x,t}}
 \lesssim_\epsilon h^{-\epsilon}
 \left\|\sum_{\theta,j}|a_{\theta,j}|^2
       |p_{\theta,h}(x-z_{\theta,j})|^2\right\|_{L^{3/2}_x}.
\end{equation}
The arbitrary small loss has been renamed after squaring. Approximate nonnegative smooth compactly supported densities by finite point masses and then use monotone approximation and density. Combining \eqref{packetlower}, \eqref{packetupper}, and the dual of \eqref{spatialmax} proves
\begin{equation}\label{transport}
 \left\|\sum_\theta (A_{\theta,h}*g_\theta)(x+t\omega_\theta)
       \right\|_{L^{3/2}(\R^3\times[-1,1])}
 \lesssim_\epsilon h^{-\epsilon}\left\|\sum_\theta g_\theta\right\|_{3/2},
\end{equation}
where $A_{\theta,h}=|B_{\omega_\theta,h}|^{-1}\1_{B_{\omega_\theta,h}}$. The duality used here is explicit: for $a\geq0$,
\[
 \int a\sum_\theta B_\theta*g_\theta
 \leq\int\left(\sum_\theta g_\theta\right)
               \sup_\theta (\widetilde B_\theta*a),
\]
where $B_\theta$ is the Schwartz kernel in \eqref{packetupper}.

To pass from \eqref{transport} to spacetime convolution, write
\[
 \sum_\theta K_{\theta,h}*G_\theta(x,t)
 \asymp\int_{|t-s|\leq1}\sum_\theta
 (A_{\theta,h}*G_\theta(\cdot,s))(x+(t-s)\omega_\theta)\,ds.
\]
On $t\in[m,m+1]$, only $s\in[m-1,m+2]$ occurs. Minkowski, \eqref{transport}, and H\"older in $s$ bound the $L^{3/2}$ norm on this slab by
\[
 C_\epsilon h^{-\epsilon}
 \left\|\sum_\theta G_\theta\right\|_{L^{3/2}(\R^3\times[m-1,m+2])}.
\]
Raise to the power $3/2$ and sum in $m$. These enlarged time intervals have bounded overlap. This proves \eqref{avg4}. Dyadic dilation proves the Schwartz-kernel version. Finally duality against arbitrary nonnegative $G_\theta$ gives the maximal formulation for a finite direction family; approximation gives the full supremum. Uncentred plates containing the evaluation point are contained in fixed centred enlargements.
\end{proof}

\section{Packets and their concentration}\label{packets}
We now work in the null chart of Section~\ref{nullcoordinates}, with the
propagator \eqref{nullprop}. An initial packet has two spatial widths
$\sqrt h$ and one width $h$. During a unit time interval it follows a null
ray, giving a physical plank of volume comparable to $h^2$.

Put $\delta=\sqrt h$. At unit time length, choose angular labels separated at scale $\delta$. For each label $A$, choose a spatial lattice of volume $h^2$ in the affine coordinates $(b,y-2A\cdot b)$, with side lengths $\delta,\delta,h$. A packet label is $a=(A,B,Y)$, and an $L^2$-normalized packet is
\begin{equation}\label{packetdef}
 p_a(b,y,t)=c_h\int b_0\big(\sqrt h(\eta-2\rho A),h\rho\big)
 e^{i((b-B)\cdot\eta-(y-Y)\rho-t|\eta|^2/(4\rho))}\,d\eta\,d\rho,
\end{equation}
where $c_h\asymp h$ and $b_0$ is a fixed smooth compactly supported profile with its second argument positive. Fixed finite collections of such profiles and fixed oversampling are harmless.

The core plate of this packet is
\begin{equation}\label{nullplate}
 \begin{split}
 P_a=\{(b,y,t):&\ 0\leq t\leq1,\quad |b-B-tA|\leq C\delta,\\
 &|y-Y-2A\cdot(b-B)+t|A|^2|\leq Ch\}.
 \end{split}
\end{equation}
It has volume comparable to $h^2$. Its spatial centre follows the ray
$(B+tA,Y+t|A|^2)$. The factor $c_h\asymp h$ in \eqref{packetdef} follows from Plancherel: the spatial frequency box has volume comparable to $h^{-2}$. Consequently the packet is bounded by $C h^{-1}$ on its core and has rapid decay in the normalized coordinates of \eqref{nullplate}. For finite coefficients $d_a$, define
\begin{equation}\label{density}
 e(d)=\sum_a|d_a|^2,\qquad
 D_d=h^{-2}\sum_a|d_a|^2\1_{P_a}.
\end{equation}

Use the tests of Section~\ref{preliminaries} at $T=0$ and $\ell=1$,
so that their uncertainty condition is $rf\geq\delta=\sqrt h$. For a
finite coefficient array define its \emph{centre capacity} by
\begin{equation}\label{capacity}
 \kappa_\gamma(d)=\sup_E\frac{\sum_{a\in E}|d_a|^2}{w(E)}.
\end{equation}
Here $a\in E$ means that the initial centre $(A,B,Y)$ satisfies the test;
$w(E)$ is given by \eqref{testweight}. Thus capacity measures the largest
packet energy per allowed test weight. It is a quantity attached to the
positive array $|d_a|^2$, not yet to the possibly cancelling wave
$\sum_a d_ap_a$. Comparing these two objects is a separate step.

\section{Swept volume and decomposition by capacity}
\begin{lemma}[Swept volume]\label{volume}
For any test $E$ above,
\begin{equation}\label{swept}
 \left|\bigcup_{a\in E}P_a\right|\leq C r^4f\leq Cw(E).
\end{equation}
The constant is independent of $r,f,h$ and of the centre of the test.
\end{lemma}

\begin{proof}
In normalized coordinates the rays are
$x(t)=x_0+tV$, $z(t)=z_0+t|V|^2$. Direct differentiation gives
\begin{equation}\label{identities}
 L_V^2F=0,\qquad
 |\nabla_VG|^2=1+4(\beta+kt)G-4kF,\qquad
 \nabla_VG=\nabla_xF+2V\partial_zF.
\end{equation}
Thus $G$ is constant on each ray and $F(t)=F(0)+tG(0)$. On a selected ray, $|F(t)|\leq2f$, $|G(t)|\leq f$, and $x,z,V$ remain bounded for $0\leq t\leq1$.

For $f$ below a sufficiently small fixed constant, \eqref{identities} implies
$|\nabla_xF|+|\partial_zF|\geq c_0>0$. At each fixed $t$, the set in a fixed bounded $(x,z)$ box where $|F|\leq2f$ and this gradient lower bound holds has volume $O(f)$. To see this without any regularity assumption on the boundary, partition by a coordinate derivative of magnitude at least a fixed constant. Along the corresponding coordinate, a quadratic has at most two monotonic pieces; the preimage of an interval of length $O(f)$ on each piece has length $O(f)$. Fubini proves the assertion. For larger $f$ the bounded box has volume $O(1)=O(f)$.

Packet thickness preserves this argument. A core displacement has the form
\[
 \Delta x=O(\delta/r),\qquad
 \Delta z=2V\cdot\Delta x+O(h/r^2).
\]
Since $rf\geq\delta$, these perturbations change $F$ and $G$ by $O(f)$; the gradient lower bound survives when $f$ is small enough. The normalized spatial volume is consequently $O(f)$. The spatial Jacobian is $r^4$, and the time interval has length one. This proves \eqref{swept}. The round case follows directly from the bounded normalized box.
\end{proof}

The following proposition is an adaptation of Lemma 4.2 of \cite{Diagonal}.

\begin{proposition}[Finite-array capacity decomposition]\label{decomp}
Every nonzero finite array $d$, with $n=\#\supp d$, has a decomposition into disjoint coordinate subarrays
\[
 d=\sum_{i=1}^J d^{(i)},\qquad J\leq C\log(2+n),
\]
such that
\begin{equation}\label{capacitysum}
 \sum_i e(d^{(i)})\kappa_\gamma(d^{(i)})^{1/2}
 \leq C\int_{\R^2\times\R\times[0,1]}D_d^{3/2}.
\end{equation}
\end{proposition}

\begin{proof}
The uncertainty condition gives $w(E)\geq h^2$. A round test with $r=C\delta$ captures any one centre and has weight $O(h^2)$. Writing $D=\max_a|d_a|^2$, we have
\[
 cD/h^2\leq\kappa_\gamma(d)\leq nD/h^2.
\]
Set $A=D/(h^2n^2)$. Any coordinate remainder $q$ with capacity at most $A$ satisfies
\[
 e(q)\sqrt{\kappa_\gamma(q)}\leq D^{3/2}/h
 \lesssim\int D_d^{3/2},
\]
using the contribution of a largest original coefficient on its own core plate.

Starting with $d$, whenever the remaining capacity exceeds $A$, choose a test realizing at least half its supremum, and remove all remaining coordinates in that test. If the removed bundle is $q_j$, put
$\rho_j=e(q_j)/w(E_j)$. Then $\rho_j\geq\kappa_\gamma(\text{remainder})/2$. This procedure terminates after at most $n$ steps. Group the removed bundles according to the dyadic interval containing the capacity of the remainder from which they were removed. There are $O(\log(2+n))$ nonempty groups. Monotonicity of capacity under coordinate restriction shows that the cost of one group is at most
\[
 C\sum_{j\text{ in group}}e(q_j)\sqrt{\rho_j}.
\]
By \eqref{swept} and H\"older,
\[
 e(q_j)\sqrt{\rho_j}
 =\frac{e(q_j)^{3/2}}{\sqrt{w(E_j)}}
 \lesssim\int D_{q_j}^{3/2}.
\]
The bundles are disjoint and $D_d=\sum_jD_{q_j}+D_{\mathrm{remainder}}$. Since $u\mapsto u^{3/2}$ is superadditive on nonnegative numbers, summing proves \eqref{capacitysum}, including the remainder as one further group.
\end{proof}

\section{From centre capacity to wave regularity}\label{sourcematch}
We next derive the packet estimate supplied by the propagation input. This is the second use of the local-smoothing article.

\begin{proposition}[Regularity-controlled synthesis]\label{wavehyp}
Assume External inputs~\ref{enlargedinput} and~\ref{propinput}.
Fix a packet system \eqref{packetdef}. On every fixed polynomial complexity range for the label support, for every $0<\gamma<1$ and $\zeta>0$,
\begin{equation}\label{waveinput}
 \left\|\sum_a d_ap_a\right\|_{L^3(\R^3\times[0,1])}^3
 \leq C_{\gamma,\zeta}h^{-100\gamma-\zeta}
               e(d)\sqrt{\kappa_\gamma(d)}.
\end{equation}
The bound is uniform in coefficient sizes and in coordinate restrictions, with the fixed phase, window, and complexity constants held fixed.
\end{proposition}

In the round-cone theorem the regularity condition is imposed on the \emph{analyzed wave state}, while \eqref{capacity} concerns packet centres. We now compare these two quantities.

\begin{proof}
Let $V_1$ be the plateau analyzer of \cite[Definition 8.1]{OAI}, and put $v_0=\sum_a d_ap_a(\cdot,0)$. We first prove, uniformly over all the stated tests,
\begin{equation}\label{checkpoint}
 \int_E|V_1v_0|^2\,d\nu
 \leq C\kappa_\gamma(d)w(E),
 \qquad \|v_0\|_2^2\leq C e(d),
\end{equation}
with negligible tails and fixed polynomial bounds, also after restricting the coefficient array. The first inequality requires an estimate for the analyzer applied to the entire synthesized state.

Use the analyzer \eqref{analyzer} with $\ell=1$ and write $d\nu=d\nu_1$.
It is isometric by \eqref{isometry}. For the packet with label $a=(A,B,Y)$, make the change $\eta=2\rho A+h^{-1/2}\zeta$, $\rho=h^{-1}s$ in its analyzed Fourier integral. The additional angular factor becomes
$\psi(\zeta/(2s)-(A'-A)/\delta)$. It vanishes unless $|A'-A|\leq C\delta$. Its derivatives in $(\zeta,s)$ are uniformly bounded there. Integration by parts gives, for every $L$,
\begin{equation}\label{analysisatom}
 |V_1p_a(A',b,y)|\leq C_Lh^{-1}\1_{|A'-A|\leq C\delta}
 \left(1+\frac{|b-B|}{\delta}
 +\frac{|y-Y-2A\cdot(b-B)|}{h}\right)^{-L}.
\end{equation}
Denote the indicator times the decreasing weight on the right by $W_a$. Taking its exponent sufficiently large, the packet lattice gives $\sum_aW_a\leq C$: only a bounded number of angular labels are relevant to a fixed $A'$, and for each such label the spatial lattice sum is bounded in its own affine coordinates. Weighted Cauchy--Schwarz consequently gives
\[
 |V_1v_0|^2\leq C_L h^{-2}\sum_a|d_a|^2W_a.
\]
Since $\int W_a\,d\nu\lesssim h^2$, integration over all row space proves the energy bound in \eqref{checkpoint}.

For integration over a test $E$, dyadically sort the packet centres by their distance from $E$ in their angular and tilted spatial packet coordinates. The normalized tails of $h^{-2}W_a\,d\nu$ decrease faster than any prescribed power. Hence
\begin{equation}\label{besselcheck}
 \int_E\left|\sum_a d_a V_1p_a\right|^2d\nu
 \lesssim_L\sum_{j\geq0}2^{-jL}
                  \sum_{a\in E^{(2^j)}}|d_a|^2,
\end{equation}
where $E^{(P)}$ enlarges a test by $P$ packet widths in position and angle. To justify the use of the external enlarged-test lemma on this set, write the perturbations in normalized test coordinates. They satisfy
\[
 |\Delta V|+|\Delta x|\lesssim P\delta/r,\qquad
 |\Delta z-2V\cdot\Delta x|\lesssim P^2h/r^2.
\]
Because $\delta/r\leq f$, substitution in the explicit quadratic polynomials $F,G$ shows that the enlarged set has normalized round bounds at most $P^C$ and sublevel bounds at most $P^Cf$. The discriminant of the defining polynomial is still one. The covering argument of \cite[Lemma 8.4]{OAI} therefore covers it by admissible tests of total weight at most $CP^{C'}w(E)$. This covering argument is geometric and applies to the positive atomic measure $\sum_a|d_a|^2\delta_a$, as well as to analyzed row measures. Thus
\[
 \sum_{a\in E^{(P)}}|d_a|^2\leq CP^{C'}\kappa_\gamma(d)w(E).
\]
Choose $L>C'$ in \eqref{besselcheck}. This proves \eqref{checkpoint}. In particular, the mixed angular--transverse normal displacement is accounted for at scale $h$; no isotropic enlargement has been substituted for it.

At terminal length $h$, assign analyzed row masses $m_{j,Q}$ to spatial $h$-cubes. The assignments are admissible round tests with weights comparable to one: the angular chart is bounded and a fixed round width contains all its angles. The assignments are disjoint in row space at each time. External input~\ref{propinput} gives
\[
 h^{-1/2}\sum_{j,Q}m_{j,Q}^{3/2}
 \lesssim h^{-100\gamma-\zeta}\|v_0\|_2^2\sqrt{\kappa},
\qquad \kappa=C\kappa_\gamma(d).
\]
Here is a direct terminal synthesis bound, which does not require that $v_0$ originate from the particular physical input in Section 10 of the source. On $[jh,(j+1)h]$, the kernel of $U_0(t-jh)V_h^*$ is bounded by
\[
 C_Lh^{-3}(1+|z-z'|/h)^{-L},
\]
uniformly over the bounded terminal angular variables. This follows by scaling the compact frequency band by $h$; the elapsed time is at most $h$. Weighted Cauchy--Schwarz gives, for $z\in Q$,
\[
 |U_0(t)v_0(z)|^2\lesssim h^{-3}
 \sum_{Q'}(1+\operatorname{dist}(Q,Q')/h)^{-L}m_{j,Q'}.
\]
The spacetime cell volume is $h^4$. Raise to $3/2$, integrate, and use the summability of the discrete kernel to obtain
\[
 \int_0^1\int |U_0(t)v_0|^3
 \lesssim h^{-1/2}\sum_{j,Q}m_{j,Q}^{3/2}.
\]
The exponent theorem initially applies after a fixed polynomial truncation. Packet decay makes the omitted tails negligible to any prescribed power; the truncation exponent is fixed before invoking that theorem.

We check its other admissibility conditions. For $n\leq Ch^{-B}$ centres in a fixed polynomial position range, \eqref{analysisatom} gives polynomial row-density bounds and negligible phase tails relative to $e(d)$. Also $\kappa_\gamma(d)\leq e(d)/h^2$ and $\kappa_\gamma(d)\geq c e(d)/(nh^2)$. If the actual wave energy $e_0$ is at least $h^M e(d)$ for a fixed $M$, all required ratios and tails relative to $e_0$ have fixed polynomial bounds. If $e_0<h^M e(d)$, the elementary band estimate gives
\[
 \int_0^1\int|U_0(t)v_0|^3\lesssim h^{-3/2}e_0^{3/2}
 \leq h^{(3M-3)/2}e(d)^{3/2}.
\]
Taking $M>B+10$ makes this at most $C e(d)\sqrt{\kappa_\gamma(d)}$. Thus tiny synthesized energy is harmless; no lower bound on individual coefficients is imposed. All bounds are inherited by coordinate restrictions. Applying the round-cone exponent bound with these fixed setup constants proves \eqref{waveinput}, after renaming the arbitrary loss.
\end{proof}

\section{From regularity-controlled synthesis to packet density}
\begin{proposition}\label{packetdensity}
Under \eqref{waveinput}, on every fixed polynomial label range,
\begin{equation}\label{packetdensityeq}
 \left\|\sum_a d_ap_a\right\|_{L^3(\R^3\times[0,1])}
 \lesssim_\epsilon h^{-\epsilon}\|D_d^{1/2}\|_3.
\end{equation}
\end{proposition}
\begin{proof}
Use Proposition~\ref{decomp} and write $v_i=\sum_a d_a^{(i)}p_a$. The triangle inequality followed by H\"older gives
\[
 \left\|\sum_i v_i\right\|_3^3
 \leq J^2\sum_i\|v_i\|_3^3
 \lesssim J^2h^{-100\gamma-\zeta}
       \sum_i e(d^{(i)})\sqrt{\kappa_\gamma(d^{(i)})}.
\]
Apply \eqref{capacitysum}. In a fixed polynomial setup, $n\leq Ch^{-B}$, hence $J^2\lesssim_B(1+\log h^{-1})^2$. Choose $\gamma,\zeta$ sufficiently small relative to the prescribed $\epsilon$, and absorb the logarithm.
\end{proof}

One can remove the position-range restriction without allowing the complexity exponent to grow. Partition initial packet centres into fixed unit spatial cubes. Each cube contains at most a fixed power of $h^{-1}$ labels, and translation gives one fixed complexity range. On an output unit cube, only a bounded number of nearby input cubes contribute to the principal part because the velocities are bounded. Apply \eqref{packetdensityeq} to these arrays and sum cubes. For distant cubes, integration by parts in \eqref{packetdef} gives
\[
 \|v_k\|_{L^3(Q_j\times[0,1])}
 \leq C_Nh^N(1+|j-k|)^{-N}e(d^{(k)})^{1/2}.
\]
Discrete Young's inequality bounds the sum of these tails by
$C_Nh^N(\sum_ke(d^{(k)})^{3/2})^{1/3}$. The core plates from one input cube lie in a fixed-volume region, so H\"older gives
$e(d^{(k)})^{3/2}\lesssim\int D_{d^{(k)}}^{3/2}$. Positivity then controls the tails by the right side of \eqref{packetdensityeq}. This proves a bound uniform over arbitrary finite position supports. Soft plate weights may replace the cores by enlarging the right side.

\section{Thickened-cone functions and time origins}\label{thickened}
We record how the packet estimate acts on arbitrary functions in the thickened cone. This step is essential: these functions are not required to be free waves.

For a sector centred at $A_\theta$, use affine frequency coordinates
\[
 \zeta=\sqrt h(\eta-2\rho A_\theta),\qquad
 s=h\rho,\qquad
 q=\sigma+A_\theta\cdot\eta-|A_\theta|^2\rho.
\]
The sector lies in a fixed box in $(\zeta,s,q)$, with $s$ bounded away from zero, since
\begin{equation}\label{residual}
 -\frac{|\eta|^2}{4\rho}
       +A_\theta\cdot\eta-|A_\theta|^2\rho
 =-\frac{|\zeta|^2}{4s}.
\end{equation}
Choose a reproducing multiplier
\[
 m_\theta(\eta,\rho,\sigma)
 =b_\theta(\zeta,s)\chi\left(\sigma+\frac{|\eta|^2}{4\rho}\right)
\]
equal to one on the sector. Its inverse Fourier transform is a constant times
$h^{-1}\psi(t)p_\theta(b,y,t)$, where $\psi$ is Schwartz and $p_\theta$ is an $L^2$-normalized packet of the form \eqref{packetdef}.

Expand $\widehat f_\theta$ in a Fourier series in a fixed larger affine box, then multiply the series by $m_\theta$. The resulting identity is
\begin{equation}\label{frameexp}
 f_\theta=\sum_{v,k} a_{\theta,v,k}\,
       \psi(t-k)p_{\theta,v}(b,y,t-k),
\end{equation}
with fixed oversampling if needed. The spatial lattice has the side lengths $\delta,\delta,h$ in the tilted coordinates, and the time lattice has constant spacing. Fixed time spacing is denoted by one; the identical argument applies after a fixed rescaling. A time translate also shifts the centre along the central ray, which is included in the spatial label $v$. The Fourier coefficients are fixed constants times $h$ times samples of $f_\theta$ at the corresponding affine physical lattice. To check this normalization, the frequency box has volume comparable to $h^{-2}$, so its Fourier-series coefficient is a constant times $h^2 f_\theta$ at the dual lattice point. Multiplication by $m_\theta$ supplies the additional factor $h^{-1}$ in its inverse transform. The resulting coefficient is therefore a constant times $h f_\theta$ at that point. The physical lattice cell has volume comparable to $h^2$; its time spacing is fixed, and its spatial sides are $\delta,\delta,h$.

The use of Fourier series is useful here because the coefficients are computed from $f_\theta$ itself. Thus interference among its Fourier components is retained before any positive upper bound is taken.

For a sufficiently large fixed $M$, put
\begin{align*}
 X_{\theta,v,k}&=\frac{b-B-(t-k)A_\theta}{\delta},\\
 Z_{\theta,v,k}&=\frac{y-Y-2A_\theta\cdot(b-B)+(t-k)|A_\theta|^2}{h},\\
 W_{\theta,v,k}(b,y,t)&=
 (1+|X_{\theta,v,k}|+|Z_{\theta,v,k}|)^{-M}(1+|t-k|)^{-M}.
\end{align*}
Let
\[
 H_a=h^{-2}\sum_{\theta,v,k}|a_{\theta,v,k}|^2W_{\theta,v,k},
 \qquad K_\theta=h^{-2}W_{\theta,0,0}.
\]
Band-limited reproduction and Cauchy--Schwarz give the sample estimate
\[
 |a_{\theta,v,k}|^2
 \lesssim\int |f_\theta(z)|^2W^{\mathrm{fast}}_{\theta,v,k}(z)\,dz,
\]
where the exponent in $W^{\mathrm{fast}}$ is chosen larger than $M+5$. Convolution of decreasing lattice weights yields
\begin{equation}\label{analysisdensity}
 H_a\lesssim\sum_\theta |f_\theta|^2*K_\theta.
\end{equation}
All these statements follow in the affine coordinates from ordinary Fourier series and the estimate
$\sum_{n\in\Z^4}\ip{u-n}^{-L}\ip{n-v}^{-M}\lesssim\ip{u-v}^{-M}$ for $L>M+4$. They retain cancellation within each original sector.

For completeness, the different time origins in \eqref{frameexp} do not require a new propagation hypothesis. On $t\in[j,j+1]$, re-expand the wave originating at $k$ in the frame based at $j$. After removing the linear phase, its transition matrix consists of Fourier coefficients of
\[
 b(\zeta,s)\exp\left(-i(j-k)\frac{|\zeta|^2}{4s}\right).
\]
By integration by parts its entries satisfy, for every $L$,
\[
 |M_{j-k}(u)|\leq C_L(1+|j-k|)^{C_L}(1+|u|)^{-L}.
\]
The constants in this formula are independent of $h$, by the exact identity \eqref{residual}. More explicitly, let
\[
 S_k=h^{-2}\sum_{\theta,v}|a_{\theta,v,k}|^2
           (1+|X_{\theta,v,k}|+|Z_{\theta,v,k}|)^{-M},
 \qquad V_k=\sum_{\theta,v}a_{\theta,v,k}p_{\theta,v}(b,y,t-k).
\]
Then $H_a=\sum_k(1+|t-k|)^{-M}S_k$. In the rebased lattice, weighted Cauchy--Schwarz bounds each new squared coefficient by
\[
 C_L(1+|j-k|)^{C_L}\sum_v
       |a_{\theta,v,k}|^2(1+|u-v|)^{-L}.
\]
The index $u$ is measured after translating the old centre by its ray displacement to time $j$. Fixed lattice offsets do not affect the bound. If a rebased core contains $(b,y,t)$, its distance from an old ray in normalized packet coordinates is bounded by $C(1+|u-v|)$. Summing the last display against rebased cores, with $L>M+4$, therefore bounds their density by $C(1+|j-k|)^C S_k$ on the slab. Applying \eqref{packetdensityeq} to the rebased array gives
\[
 \|V_k\|_{L^3(\R^3\times[j,j+1])}
 \lesssim_\epsilon h^{-\epsilon}(1+|j-k|)^C
                 \|S_k^{1/2}\|_{L^3(\R^3\times[j,j+1])}.
\]
The fixed synthesis profiles are chosen with a plateau on the old profiles, so all constants in the packet estimate are independent of $j,k$; only the displayed transition matrix depends on $j-k$.

Choose $A>1$ and put $q_{j-k}=(1+|j-k|)^{-A}$. Schwartz decay of $\psi$ absorbs the polynomial factor and an additional factor $(1+|j-k|)^{-M/2}$. Minkowski followed by
$(\sum q_k u_k)^3\leq(\sum q_k)^2\sum q_k u_k^3$
thus bounds the cube of the norm on this slab by
\[
 C_\epsilon h^{-3\epsilon}\int_{\R^3\times[j,j+1]}
 \sum_k q_{j-k}\big((1+|t-k|)^{-M}S_k\big)^{3/2}
 \leq C_\epsilon h^{-3\epsilon}\int_{\R^3\times[j,j+1]}H_a^{3/2}.
\]
Summing the disjoint slabs gives
\begin{equation}\label{fullsynthesis}
 \left\|\sum_{\theta,v,k}a_{\theta,v,k}\psi(t-k)
                       p_{\theta,v}(b,y,t-k)\right\|_3
 \lesssim_\epsilon h^{-\epsilon}\|H_a^{1/2}\|_3.
\end{equation}
Finite truncation followed by convergence proves the statement for the Schwartz functions under consideration. Fixed oversampling splits into finitely many translates of the displayed lattices.

\section{Proof of the cone estimate}
By \eqref{fullsynthesis} and \eqref{analysisdensity},
\[
 \left\|\sum_\theta f_\theta\right\|_3
 \lesssim_\epsilon h^{-\epsilon}
 \left\|\sum_\theta |f_\theta|^2*K_\theta\right\|_{3/2}^{1/2}.
\]
The null-coordinate plate kernels $K_\theta$ are fixed linear images of the conic plate kernels in Proposition~\ref{avg}, with comparable dimensions on the compact angular chart. That proposition therefore implies
\[
 \left\|\sum_\theta f_\theta\right\|_3
 \lesssim_\epsilon h^{-C\epsilon}
 \left\|\sum_\theta|f_\theta|^2\right\|_{3/2}^{1/2}.
\]
Renaming the arbitrary loss, returning to $R=h^{-1}$, and summing finitely many angular charts proves Theorem~\ref{main}. Fixed smooth chart multipliers act boundedly on $L^3(\ell^2)$ by their integrable convolution kernels, so this last step introduces no additional square-function assumption.

\section{Elliptic estimates from frequency slices}\label{ellipticsection}
The following consequence also explains why the exponent $3$ is natural.
A transverse slice of the four-dimensional cone is a two-dimensional elliptic
surface in three-dimensional frequency space.

\begin{corollary}\label{ellipticmain}
Assume \eqref{target}. Let $\Sigma$ be a sphere, an ellipsoid, or a compact
patch of an elliptic paraboloid in $\R^3$. Let $\vartheta$ be a regular
finite-overlap cap family for its $R^{-1}$-neighbourhood, with tangential
widths $R^{-1/2},R^{-1/2}$ and normal width $R^{-1}$.
For every $\varepsilon>0$ and Schwartz family with
$\supp\widehat g_\vartheta\subset\vartheta$,
\begin{equation}\label{elliptictarget}
 \left\|\sum_\vartheta g_\vartheta\right\|_{L^3(\R^3)}
 \leq C_{\varepsilon,\Sigma}R^\varepsilon
 \left\|\left(\sum_\vartheta|g_\vartheta|^2\right)^{1/2}\right\|_{L^3(\R^3)}.
\end{equation}
In particular, the external inputs of Theorem~\ref{main} imply this estimate.
\end{corollary}

In this section all frequencies are at unit scale, and $\delta=R^{-1}$. A cone sector has normal thickness $O(\delta)$ and angular width $O(\sqrt\delta)$. An elliptic cap has two tangential widths $O(\sqrt\delta)$ and normal thickness $O(\delta)$.

We first note two harmless changes in the formulation of \eqref{target}. The radial interval $[1,2]$ may be replaced by any fixed compact subinterval of $(0,\infty)$, and a fixed invertible linear change of frequency coordinates is permitted. For the first assertion, use finitely many dilated annular charts. Their smooth cutoffs have uniformly integrable physical kernels and act boundedly on $L^3(\ell^2)$. For the second, apply the dual linear change in physical variables; its determinant contributes the same factor to both sides of the square function estimate. Fixed enlargement and bounded duplication of sectors cost only fixed constants.

\subsection{A tensor-product restriction lemma}

\begin{lemma}\label{tensorlemma}
Let $\{\theta\}$ be a family of frequency regions in $\R^{n+1}$ satisfying
\begin{equation}\label{ambientSF}
 \left\|\sum_\theta F_\theta\right\|_{L^p(\R^{n+1})}
 \leq A(\delta)
 \left\|\left(\sum_\theta|F_\theta|^2\right)^{1/2}\right\|_{L^p(\R^{n+1})},
 \qquad 1\leq p<\infty.
\end{equation}
Suppose regions $\vartheta\subset\R^n$ satisfy
\[
 \vartheta\times[u_0-a\delta,u_0+a\delta]\subset\theta(\vartheta)
\]
for a fixed $a>0$, where the assignment to ambient regions has bounded multiplicity. Then functions with $\supp\widehat g_\vartheta\subset\vartheta$ satisfy
\begin{equation}\label{sliceSF}
 \left\|\sum_\vartheta g_\vartheta\right\|_{L^p(\R^n)}
 \lesssim A(\delta)
 \left\|\left(\sum_\vartheta|g_\vartheta|^2\right)^{1/2}\right\|_{L^p(\R^n)}.
\end{equation}
\end{lemma}

\begin{proof}
Choose a nonzero Schwartz function $\psi$ with $\supp\widehat\psi\subset[-a,a]$ and define
\[
 F_\vartheta(x,s)=g_\vartheta(x)e^{iu_0s}\psi(\delta s).
\]
Its Fourier transform is
\[
 \widehat F_\vartheta(\xi,u)
 =\delta^{-1}\widehat g_\vartheta(\xi)
                    \widehat\psi\big((u-u_0)/\delta\big),
\]
so it has the required ambient support. First assume the assignment is injective and apply \eqref{ambientSF}. Fubini gives the exact identities
\begin{align*}
 \left\|\sum_\vartheta F_\vartheta\right\|_p
 &=\delta^{-1/p}\|\psi\|_p\left\|\sum_\vartheta g_\vartheta\right\|_p,\\
 \left\|\left(\sum_\vartheta|F_\vartheta|^2\right)^{1/2}\right\|_p
 &=\delta^{-1/p}\|\psi\|_p
       \left\|\left(\sum_\vartheta|g_\vartheta|^2\right)^{1/2}\right\|_p.
\end{align*}
Cancel the common nonzero factor. For bounded multiplicity, split the indices into a bounded number of injective families and use the triangle inequality. This proves \eqref{sliceSF}. In particular, the passage to a slice introduces no power of $\delta$.
\end{proof}

\subsection{The sphere}

It suffices by dilation to consider the sphere $|\xi|=\lambda_0$, where $1<\lambda_0<2$ is fixed. Its $\delta$-neighbourhood is covered by caps of angular width $\sqrt\delta$. If
\[
 \big||\xi|-\lambda_0\big|\leq C\delta,
 \qquad |\tau-\lambda_0|\leq a\delta,
\]
then
\[
 |\tau-|\xi||\leq(C+a)\delta.
\]
The angular coordinate $\xi/|\xi|$ is unchanged by the extra frequency $\tau$. Each lifted elliptic cap is therefore contained in a fixed enlargement of a canonical cone sector, with bounded multiplicity. Apply Lemma~\ref{tensorlemma} with $p=3$, $u_0=\lambda_0$, and $A(\delta)=C_\varepsilon\delta^{-\varepsilon}$. This proves \eqref{elliptictarget} for the sphere. The required lift is explicitly
\[
 F_\vartheta(x,t)=g_\vartheta(x)e^{i\lambda_0t}\psi(t/R).
\]
Both norms acquire the factor $R^{1/3}\|\psi\|_3$, which cancels exactly.

\subsection{The paraboloid}

Use the null-frequency form of the cone
\begin{equation}\label{nullconesection}
 \mathcal C_0=\{(\eta,\rho,\sigma):\rho>0,\quad
                         \sigma=-|\eta|^2/(4\rho)\}.
\end{equation}
This is a linear image of the light cone. Indeed, if $u=\rho-\sigma$ and $v=\rho+\sigma$, then
\[
 u^2-v^2=-4\rho\sigma=|\eta|^2,
\]
and $u>0$ on the indicated branch. The slice $\rho=1$ is the elliptic paraboloid
\[
 \Pi=\{(\eta,\sigma):\sigma=-|\eta|^2/4\}\subset\R^3.
\]
Restrict $\eta$ to a fixed compact set. Let $\vartheta$ have centre $\eta_\vartheta$ and support conditions
\[
 |\eta-\eta_\vartheta|\leq C\sqrt\delta,
 \qquad \left|\sigma+\frac{|\eta|^2}{4}\right|\leq C\delta.
\]
For $|\rho-1|\leq a\delta$ and sufficiently small $\delta$,
\begin{equation}\label{parabolaliftsupport}
 \left|\sigma+\frac{|\eta|^2}{4\rho}\right|
 \leq \left|\sigma+\frac{|\eta|^2}{4}\right|
       +\frac{|\eta|^2}{4}|\rho^{-1}-1|
 \lesssim\delta.
\end{equation}
Moreover the angular coordinate of \eqref{nullconesection} satisfies
\[
 \left|\frac{\eta}{2\rho}-\frac{\eta_\vartheta}{2}\right|
 \lesssim\sqrt\delta.
\]
Thus the product of an elliptic cap with this $\rho$ interval lies in a fixed enlarged cone sector. The radial variable stays in a fixed compact positive interval. These observations verify all the support hypotheses of Lemma~\ref{tensorlemma}.

In the physical convention used earlier, the variables $(b,y,t)$ are paired with $(\eta,-\rho,\sigma)$. Accordingly, for $g_\vartheta(b,t)$ the lift is
\begin{equation}\label{parabolalift}
 F_\vartheta(b,y,t)=g_\vartheta(b,t)e^{-iy}\psi(y/R).
\end{equation}
Applying the cone estimate and cancelling $R^{1/3}\|\psi\|_3$ proves the elliptic square function estimate on $\Pi$. Bounded $R$ is absorbed by the trivial finite-cap estimate.

Every positive-definite quadratic graph
\[
 \sigma=c+\ell\cdot\eta+\tfrac12\eta^{\mathsf T}H\eta,
 \qquad H=H^{\mathsf T}>0,
\]
is a fixed invertible affine image of $\Pi$, allowing reflection of the last coordinate. The dual physical change has the same Jacobian factor on both sides; a frequency translation is a common modulation of modulus one. Regular cap dimensions are preserved up to fixed constants. This proves the paraboloid assertion and the affine-invariance assertion in Corollary~\ref{ellipticmain}.

\begin{remark}
The slicing proof preserves the exponent and the arbitrary positive power
loss; it does not give the estimate with $\varepsilon=0$. A general elliptic
surface is not an affine slice of this quadratic cone. For such a surface,
the same tensor-product argument requires the square function estimate for
its own homogeneous cone. That additional estimate is not asserted here.
\end{remark}

\end{document}